\documentclass[12pt,twoside]{amsart}
\usepackage{amsmath}
\usepackage{amsthm}
\usepackage{amsfonts}
\usepackage{amssymb}
\usepackage{latexsym}
\usepackage{mathrsfs}
\usepackage{amsmath}
\usepackage{amsthm}
\usepackage{amsfonts}
\usepackage{amssymb}
\usepackage{latexsym}
\usepackage{geometry}
\usepackage{dsfont}
\usepackage[dvips]{graphicx}
\usepackage[colorlinks=true,linkcolor=red,citecolor=blue]{hyperref}
\usepackage{color}
\usepackage[all]{xy}

\date{}
\allowdisplaybreaks[4] \footskip=15pt
\renewcommand{\uppercasenonmath}[1]{}

\numberwithin{equation}{section} \theoremstyle{plain}
\usepackage{graphicx,amssymb}
\usepackage[all]{xy}
\usepackage{amsmath}

\allowdisplaybreaks
\usepackage{amsthm}
\usepackage{color}

\theoremstyle{plain}

\theoremstyle{plain}
\newtheorem{theorem}{Theorem}[section]
\newtheorem{proposition}[theorem]{Proposition}
\newtheorem{lemma}[theorem]{Lemma}
\newtheorem{corollary}[theorem]{Corollary}
\newtheorem{example}[theorem]{Example}
\newtheorem*{open question}{Open Question}
\newtheorem{definition}[theorem]{Definition}

\theoremstyle{definition}
\newtheorem*{acknowledgement}{Acknowledgement}

\theoremstyle{remark}
\newtheorem{remark}[theorem]{Remark}

\newcommand{\A}{\mathcal{A}}

\newcommand{\N}{\mathcal{N}}
\newcommand{\C}{\mathcal{C}}

\newcommand{\Id}{\mathrm{Id}}

\def\p{\frak p}
\def\m{\frak m}

\def\Hom{{\rm Hom}}

\def\Ker{{\rm Ker}}

\def\Im{{\rm Im}}
\def\Coker{{\rm Coker}}

\def\Jac{{\rm Jac}}

\def\Max{{\rm Max}}

\def\Spec{{\rm Spec}}
\def\Max{{\rm Max}}

\begin{document}
\begin{center}
{\large  \bf Characterizing  $S$-Artinianness by uniformity}

\vspace{0.5cm}
Xiaolei Zhang,\ \ \ Wei Qi\\
\bigskip
School of Mathematics and Statistics, Shandong University of Technology,
Zibo 255049, China\\

Corresponding author: Xiaolei Zhang, E-mail: zxlrghj@163.com\\
\end{center}

\bigskip
\centerline { \bf  Abstract}
\bigskip
\leftskip10truemm \rightskip10truemm \noindent
Let $R$ be a commutative ring with identity and $S$ a multiplicative subset of $R$. An $R$-module $M$ is said to be a $u$-$S$-Artinian  module if there is $s\in S$ such that any descending chain of submodules of $M$ is $S$-stationary with respect to $s$.  The notion of $u$-$S$-Artinian modules are characterized in terms of ($S$-MIN)-conditions and $u$-$S$-cofinite properties. We call a ring $R$ is a $u$-$S$-Artinian ring if $R$ itself is a $u$-$S$-Artinian module, and then show that any $u$-$S$-semisimple ring is $u$-$S$-Artinian. It is proved that a ring  $R$ is  $u$-$S$-Artinian if and only if  $R$ is $u$-$S$-Noetherian, the $u$-$S$-Jacobson radical $\Jac_S(R)$ of $R$ is  $S$-nilpotent and $R/\Jac_S(R)$ is a  $u$-$S/\Jac_S(R)$-semisimple ring. Besides, some examples are given to distinguish Artinian rings, $u$-$S$-Artinian rings and $S$-Artinian rings.
\\
\vbox to 0.3cm{}\\
{\it Key Words:} $u$-$S$-Artinian ring; $u$-$S$-Artinian module; $u$-$S$-Noetherian ring; $u$-$S$-semisimple ring; $u$-$S$-Jacobson radical.\\
{\it 2020 Mathematics Subject Classification:} 13E10, 13C12.

\leftskip0truemm \rightskip0truemm
\bigskip
\section{Introduction}
Throughout this article, all rings are commutative rings with identity and all modules are unitary.  A subset $S$ of $R$ is called a multiplicative subset of $R$ if $1\in S$ and $s_1s_2\in S$ for any $s_1\in S$, $s_2\in S$. Early in 2002,  Anderson and Dumitrescu \cite{ad02}  introduced the so-called $S$-Noetherian ring $R$, in which for any ideal $I$ of $R$, there exists a finitely generated ideal $K$ of $R$  such that $sI\subseteq K\subseteq I$ for some $s\in S$. Note that Cohen's Theorem, Eakin-Nagata Theorem and Hilbert Basis Theorem for $S$-Noetherian rings are also given in \cite{ad02}. The notion of $S$-Noetherian rings provides a good direction for  $S$-generalizations of other classical rings (see \cite{Ah18,ahz19,B19,bh18,kkl14} for example).  However, it is often difficult to study these  $S$-generalizations of classical rings via a module-theoretic approach. The essential difficulty is that the selected element $s\in S$ is often not ``uniform'' in their definitions. To overcome this difficulty for Noetherian properties, Qi et al. \cite{QKWCZ21} recently introduced the notions of uniformly $S$-Noetherian rings which are $S$-Noetherian rings such that $s$ is independent on $I$ in the definition of $S$-Noetherian rings. They also introduced the notion of $u$-$S$-injective modules and then characterized uniformly $S$-Noetherian rings in terms of $u$-$S$-injective modules. Some other uniform $S$-versions of rings and modules, such as semisimple rings, von Neumann regular rings, projective modules and flat modules  are introduced and studied by the named authors and  coauthors in \cite{zwz21,zwz21-p}.

In 2020, Sengelen et al. \cite{stk20} introduced the notions of $S$-Artinian rings for which  any descending chain of ideals $I_1\supseteq I_2\supseteq \cdots\supseteq I_m\supseteq \cdots$ of $R$ satisfies $S$-stationary condition, i.e., then there exist $s\in S$ and $k \in  \mathbb{Z}^{+}$ such that $sI_k\subseteq I_n$ for all $n \geq  k$. In the definition of $S$-Artinian rings, it is easy to see that although the element $s\in S$ is independent on $n$ but it is certainly dependent on the given descending chain of ideals. Recently,  \"{O}zen et al.\cite{ONTK21} extended the notion of $S$-Artinian rings  to that  of $S$-Artinian modules by replacing descending chains of ideals to these of  submodules. And then, Omid et al. \cite{OA21} introduced the notion of weakly $S$-Artinian modules, for which every descending chain $N_1\supseteq N_2\supseteq \cdots\supseteq N_m\supseteq \cdots$  of submodules of $M$ is weakly $S$-stationary, i.e., there exists $k \in  \mathbb{Z}^{+}$ such that for each $n \geq  k$, $s_nN_k\subseteq N_n$  for some $s_n \in S$.  In the definition of weakly $S$-Artinian modules, it is easy to see the element $s_n$ is dependent both on  $n$ and the descending chain of ideals. So  the notion of weakly $S$-Artinian modules is certainly  a ``weak'' version of that of  $S$-Artinian modules. In this article,  we introduced and study the ``uniform'' $S$-version of Artinian rings  and modules (we call them $u$-$S$-Artinian rings  and modules) such that the element $s$ given in the definition of $S$-Artinian rings  and modules is both independent  on $n$ and the descending chain of ideals or submodules.  Obviously, we have the following implications:
$${\boxed{\mbox{Artinian rings}}}\Longrightarrow {\boxed{$u$\mbox{-}$S$\mbox{-Artinian rings}}}\Longrightarrow {\boxed{$S$\mbox{-Artinian rings}}}$$
But neither of implications can be reversed (see Example \ref{exam-not-ut-1} and  Example  \ref{exam-not-ut-2}). Denote by $\Jac_S(R)$ the $u$-$S$-Jacobson radical of a ring $R$ (see Definition \ref{usjaco}). Then
it is also worth to mention that  a ring  $R$ is  $u$-$S$-Artinian if and only if  $R$ is $u$-$S$-Noetherian, the $u$-$S$-Jacobson radical $\Jac_S(R)$ of $R$ is  $S$-nilpotent and $R/\Jac_S(R)$ is a  $u$-$S/\Jac_S(R)$-semisimple ring (see Theorem \ref{s-artinian-s-Noetherian}).

The related notions of  uniformly $S$-torsion theory  originally emerged in \cite{zwz21}, and we give a quick review below. An $R$-module $T$ is called  $u$-$S$-torsion (with respect to $s$) provided that there exists  $s\in S$ such that $sT=0$. 
A sequence $0\rightarrow A\xrightarrow{f} B\xrightarrow{g} C\rightarrow 0$ of $R$-modules is called a short $u$-$S$-exact sequence (with respect to $s$), if $s\Ker(f)=s\Coker(f)=0$, $s\Ker(g)\subseteq \Im(f)$ and $s\Im(f)\subseteq \Ker(g)$ for some $s\in S$. An $R$-homomorphism $f:M\rightarrow N$ is a $u$-$S$-monomorphism $($resp.,  $u$-$S$-epimorphism, $u$-$S$-isomorphism$)$  (with respect to $s$) provided $\Ker(f)$ is  $($resp., $\Coker(f)$ is,  both $\Ker(f)$ and  $\Coker(f)$ are$)$  $u$-$S$-torsion (with respect to $s$). Recall from \cite{zwz21-p} an $R$-module $P$ is called $u$-$S$-projective provided that the induced sequence $$0\rightarrow \Hom_R(P,A)\rightarrow \Hom_R(P,B)\rightarrow \Hom_R(P,C)\rightarrow 0$$ is $u$-$S$-exact for any $u$-$S$-exact sequence $0\rightarrow A\rightarrow B\rightarrow C\rightarrow 0$. Suppose $M$ and $N$ are $R$-modules. We say $M$ is $u$-$S$-isomorphic to $N$ if there exists a $u$-$S$-isomorphism $f:M\rightarrow N$. A family $\C$  of $R$-modules  is said to be closed under $u$-$S$-isomorphisms if $M$ is $u$-$S$-isomorphic to $N$ and $M$ is in $\C$, then $N$ is  also in  $\C$. Note that the class of $u$-$S$-projective modules is closed under $u$-$S$-isomorphisms.  One can deduce from the following \cite[Lemma 2.1]{zwz21-p} that the existence of $u$-$S$-isomorphisms of two $R$-modules is actually an equivalence relationship.

\section{uniformly  $S$-Artinian  modules}
Let $R$ be a ring, $S$ a multiplicative subset of $R$ and $M$ an $R$-module.  Suppose $M_1\supseteq M_2\supseteq \cdots\supseteq M_n\supseteq \cdots$ is a descending chain of submodules of $M$. The family $\{M_i\}_{i\in \mathbb{Z}^{+}}$ is said to be $S$-stationary (with respect to $s$) if there exists $s\in S$ and $k\in \mathbb{Z}^{+}$ such that  $sM_k\subseteq M_n$ for every $n\geq k$. And $M$ is called an $S$-Artinian module if each descending chain of submodules $\{M_i\}_{i\in \mathbb{Z}^{+}}$ of $M$ is $S$-stationary (see \cite[definition 1]{ONTK21}). Note that  in the definition of $S$-Artinian module, the element $s$ is dependent on the given descending chain of submodules. The main purpose of this section is to introduce and study a ``uniform'' version of $S$-Artinian modules.

\begin{definition} Let $R$ be a ring and $S$ a multiplicative subset of $R$.  An $R$-module $M$ is called a $u$-$S$-Artinian $($abbreviates uniformly $S$-Artinian$)$ module $($with respect to $s)$ provided that there exists $s\in S$ such that each descending chain  $\{M_i\}_{i\in \mathbb{Z}^{+}}$ of submodules of $M$ is $S$-stationary with respect to $s$.
\end{definition}

Trivially, if $0\in S$, then every $R$-module is $u$-$S$-Artinian.
If $S_1\subseteq S_2$ are  multiplicative subsets of $R$ and $M$ is $u$-$S_1$-Artinian, then  $M$ is obviously $u$-$S_2$-Artinian.  Note that Artinian modules are exactly $u$-$\{1\}$-Artinian modules. So all Artinian modules are $u$-$S$-Artinian modules for any  multiplicative set $S$. Next we give a $u$-$S$-Artinian module which is not Artinian.
\begin{example}  Let $R=\mathbb{Z}$ be the ring of integers, $p$ a prime in $R$ and $M=\mathbb{Z}_p[[x]]$ the set of all formal power series over $\mathbb{Z}_p:=\mathbb{Z}/p\mathbb{Z}$. Set $S=\{p^n\mid n\in\mathbb{N} \}$. Then $M$ is not Artinian since the descending chain $$\langle x\rangle \supseteq \langle x^2\rangle \supseteq\cdots \supseteq \langle x^n\rangle \supseteq\cdots  $$ is not stationary. However, since $M$ is obviously $u$-$S$-torsion,  $M$ is a $u$-$S$-Artinian module.
\end{example}

Let $S$ be a multiplicative subset of $R$. We always denote by $S^{\ast}=\{r\in R\mid rt\in S$ for some $t\in R\}$ and call it to be the saturation of $S$. A multiplicative set  $S$ is said to be  saturated if $S=S^{\ast}$.
Trivially, we have $S\subseteq S^{\ast}$ for all multiplicative subset $S$ of $R$.

\begin{proposition}\label{s-loc} Let $R$ be a ring, $S$ a multiplicative subset of $R$ and $M$ an $R$-module. Let $S^{\ast}$ be the saturation of $S$. Then  $M$ is a $u$-$S$-Artinian $R$-module if and only if $M$ is a $u$-$S^{\ast}$-Artinian $R$-module.
\end{proposition}
\begin{proof} Suppose $M$ is a $u$-$S$-Artinian $R$-module. Then  $M$ is trivially a $u$-$S^{\ast}$-Artinian $R$-module since $S\subseteq S^{\ast}$. Now, suppose $M$ is a $u$-$S^{\ast}$-Artinian $R$-module. Then there is $r\in S^{\ast}$ such that each descending chain of submodules $\{M_i\}_{i\in \mathbb{Z}^{+}}$ of $M$ is $S^{\ast}$-stationary with respect to $r$, i.e., there exits $k\in \mathbb{Z}^{+}$ such that  $rM_k\subseteq M_n$ for each $n\geq k$. Since $r\in S^{\ast}$,  $rt\in S$ for some $t\in R$. Note that $rtM_k\subseteq rM_k\subseteq M_n$. Hence $M$ is a $u$-$S$-Artinian $R$-module.
\end{proof}

Let $R$ be a ring, $M$ an $R$-module and  $S$ a multiplicative subset of $R$.  For any $s\in S$, there is a  multiplicative subset $S_s=\{1,s,s^2,....\}$ of $S$. We denote by $M_s$ the localization of $M$ at $S_s$. Note that $M_s\cong M\otimes_RR_s$.

\begin{lemma}\label{s-loc} Let $R$ be a ring, $S$ a multiplicative subset of $R$ and $M$ an $R$-module. If $M$ is a $u$-$S$-Artinian $R$-module, then there exists an element $s\in S$ such that $M_{s}$ is an Artinian $R_s$-module.
\end{lemma}
\begin{proof} Let $s$ be an element in $S$ such that each family of descending chain of submodules $\{M_i\}_{i\in \mathbb{Z}^{+}}$ of $M$ is $S$-stationary with respect to $s$. Let $M_1\supseteq M_2\supseteq \cdots\supseteq M_n\supseteq \cdots$ be  a descending chain of $R_s$-submodules of $M_{s}$. Considering  the natural homomorphism $f:M\rightarrow M_{s}$, we have a  a descending chain of submodules of $M$ as follows: $$f^{-1}(M_1)\supseteq f^{-1}(M_2)\supseteq \cdots\supseteq f^{-1}(M_n)\supseteq \cdots$$ There is a $k\in \mathbb{Z}^{+}$ such that  $sf^{-1}(M_k)=f^{-1}(sM_k)=f^{-1}(M_k)\subseteq f^{-1}(M_n)$ for each $n\geq k$ since $M_k$ is an $R_s$-module. Hence  $M_k=M_n$ for each $n\geq k$. Consequently, $M_s$ is an Artinian $R_s$-module.
\end{proof}

\begin{remark} The converse of Lemma \ref{s-loc} also does not hold in general. Let $R=k[[x]]$ the formal power series ring over a field $k$. Let $S=\{1,x,x^2,\dots\}$.   Then $R_S$ is a field, and so is an Artinian $R_S$-module. However, $R$ is not a $u$-$S$-Artinian $R$-module as $R$ is not Artinian (see Proposition \ref{s-artinian-regular}).
\end{remark}

A multiplicative subset $S$ of $R$ is said to satisfy the maximal multiple condition if there
exists an $s\in S$ such that $t|s$ for each $t\in S$. Both finite multiplicative subsets and the  multiplicative subsets that consist of units  satisfy the maximal multiple condition.
Certainly, $u$-$S$-Artinian modules are $S$-Artinian modules.  And the following result shows that the converse also holds for multiplicative sets satisfying maximal multiple condition.
\begin{proposition} \label{s-loc-u-noe-fini}
	Let $R$ be a ring,  $S$ a  multiplicative subset of $R$ that satisfying the maximal multiple condition, and $M$ an $R$-module. Then $M$ is a $u$-$S$-Artinian module if and only if $M$ is an  $S$-Artinian module.
\end{proposition}
\begin{proof}  If $M$ is a $u$-$S$-Artinian module, then trivially $M$ is $S$-Artinian. Let $s\in S$ such that $t|s$ for each $t\in S$. Suppose $M$ is an $S$-Artinian module and $\{M_i\}_{i\in \mathbb{Z}^{+}}$ a descending chain of submodules of $M$. Then there exists $t\in S$ such that  $tM_k\subseteq M_n$ for each $n\geq k$. So $sM_k\subseteq tM_k\subseteq M_n$ for each $i=1,\dots, n$. Hence $M$ is a $u$-$S$-Artinian module.
\end{proof}

However, the following example shows $S$-Artinian modules are not  $u$-$S$-Artinian modules in general.

\begin{example}\label{exam-not-ut-1} Let $R$ be a valuation domain whose valuation group is $G=\prod\limits_{\aleph}\mathbb{Z}$ the Hahn product of $\aleph$-copies of $\mathbb{Z}$ with lexicographic order, where $\aleph$ is an uncountable regular cardinal $($see \cite{L01}$).$ Let $S=R\setminus\{0\}$ the set of all nonzero elements of $R$. Then $R$ itself is an $S$-Artinian $R$-module but not $u$-$S$-Artinian.
\end{example}
\begin{proof}
	
	First, we claim that  $R$ is an $S$-Artinian $R$-module. Indeed, let  $I_1\supseteq I_2\supseteq \cdots\supseteq I_n\supseteq \cdots$ be a descending chain of ideals of $R$. We may assume that each $I_i$ is not equal to  $0$. Then for each $I_i$, there exists an nonzero element $r_i\in I_i$ such that $v(r_i)\in G$. Moreover, we can assume all $r_i$ satisfy  $v(r_i)\leq v(r_{i+1})$. Since $\aleph$ is an uncountable regular cardinal, $\lim\limits_{\longrightarrow}v(r_i)<\aleph$. So there is an element $x\in G$ such that $x\geq v(r_i)$ for each $i$. Suppose $v(s)=x$. Then $0\not=s\in \bigcap\limits_{i=1}^{\infty} I_i$. Hence $sI_k\subseteq I_n$ for each $n\geq k$. Consequently, $R$ is an $S$-Artinian $R$-module. Now, we claim that $R$ is not a
	$u$-$S$-Artinian $R$-module. On contrary, suppose $R$ is  $u$-$S$-Artinian. Then, by Lemma \ref{s-loc}, there is an $s\in S$ such that  $R_s$ is an artinian domain, which is exactly a field. Let $r$ be a nonzero element such that $v(r)>nv(s)$ for all non-negative integer $n$. Then $r$ is not a unit in $R_s$,  which is a contradiction.  Consequently, $R$ is not a $u$-$S$-Artinian $R$-module. $($Note that it can also be easily deduced by Proposition \ref{s-artinian-regular}.$)$
\end{proof}

\begin{lemma}\label{s-iso} Let $R$ be a ring and $S$ a multiplicative subset of $R$. Let $M$ and $N$ be $R$-modules. If $M$ is $u$-$S$-isomorphic to $N$, then $M$ is $u$-$S$-Artinian if and only if $N$ is  $u$-$S$-Artinian.
\end{lemma}
\begin{proof} Let $M$ be $u$-$S$-Artinian with respect to $s\in S$ and $f:M\rightarrow N$ a $u$-$S$-isomorphism. Suppose $N_1\supseteq N_2\supseteq \cdots\supseteq N_n\supseteq \cdots$ is a descending chain of  submodules of $N$. Then   $f^{-1}(N_1)\supseteq f^{-1}(N_2)\supseteq \cdots\supseteq f^{-1}(N_n)\supseteq \cdots$ is a descending chain of  submodules of $N$. So there is $k\in \mathbb{Z}^+$ such that $sf^{-1}(N_k)\subseteq f^{-1}(N_n)$ for any $n\geq k$. Thus $sN_k=f(sf^{-1}(N_k))\subseteq f(f^{-1}(N_n))=N_n$ for any $n\geq k$. Hence  $M$ is $u$-$S$-Artinian. Suppose $N$ is  $u$-$S$-Artinian. Then by \cite[Lemma 2.1]{zwz21-p}, there is a $u$-$S$-isomorphism $g:N\rightarrow M$. So we can show $M$ is $u$-$S$-Artinian similarly.
\end{proof}

\begin{proposition} \label{s-u-noe-s-exact}
	Let $R$ be a ring and $S$ a multiplicative subset of $R$. Let $0\rightarrow A\rightarrow B\rightarrow C\rightarrow 0$ be an $S$-exact sequence. Then $B$ is $u$-$S$-Artinian if and only if $A$ and $C$ are $u$- $S$-Artinian. Consequently, a finite direct sum $\bigoplus\limits_{i=1}^nM_i$ is $u$-$S$-Artinian if and only if each $M_i$ is $u$-$S$-Artinian $(i=1,\dots,n)$.
\end{proposition}
\begin{proof}  We can assume that  $0\rightarrow A\rightarrow B\rightarrow C\rightarrow 0$ is an exact sequence by Lemma \ref{s-iso}.  If $B$ is $u$-$S$-Artinian, then it is easy to verify $A$ and $C$ are $u$-$S$-Artinian. Now suppose $A$ and $C$ are $u$-$S$-Artinian. Let $B_1\supseteq B_2\supseteq \cdots\supseteq B_n\supseteq \cdots$ be a descending chain of  submodules of $B$. Then $$B_1\cap A\supseteq B_2\cap A\supseteq \cdots\supseteq B_n\cap A\supseteq \cdots$$ is a descending chain of  submodules of $A$, and $$(B_1+A)/A\supseteq (B_2+ A)/A\supseteq \cdots\supseteq  (B_n+A)/A\supseteq \cdots$$ is a descending chain of  submodules of $C\cong B/A$. So there is $s_1, s_2\in S$ and $k\in \mathbb{Z}^+$ such that $s_1B_k\cap A\subseteq B_n\cap A$ and $s_2(B_k+A)\subseteq B_n+A$ for any $n\geq k$. Then one can verify that $s_1s_2B_k\subseteq B_n$ for any $n\geq k$. Hence, $B$ is $u$-$S$-Artinian.
\end{proof}

Let $\p$ be a prime ideal of $R$. We say an $R$-module $M$ is \emph{$u$-$\p$-Artinian} provided that  $M$ is  $u$-$(R\setminus\p)$-Artinian. The next result gives a local characterization of Artinian modules.
\begin{proposition}\label{s-noe-m-loc-char}
	Let $R$ be a ring and $M$ an $R$-module. Then the following statements are equivalent:
	\begin{enumerate}
		\item  $M$ is Artinian;
		\item   $M$ is $u$-$\p$-Artinian for any $\p\in \Spec(R)$;
		\item   $M$ is  $u$-$\m$-Artinian for any $\m\in \Max(R)$.
	\end{enumerate}
\end{proposition}

\begin{proof}  $(1)\Rightarrow (2)\Rightarrow (3)$: Trivial.
	
	$(3)\Rightarrow (1)$: Let $M_1\supseteq M_2\supseteq \cdots\supseteq M_n\supseteq \cdots$ be a descending chain of submodules of $M$. Then , for each $\m\in\Max(R)$, there exist $s_\m\not\in \m$ and $k_{\m}\in \mathbb{Z}^{+}$ such that  $s_\m M_{k_{\m}}\subseteq M_n$ for each $n\geq k_{\m}$. Since $R$ is generated by $\{s_\m\mid \m\in\Max(R)\}$. So there is a finite subset $\{s_{\m_1},\dots,s_{\m_t}\}$ that generates $R$. Let $k=\max\{k_{\m_1},\dots,k_{\m_t}\}$. Then $M_{k}=\langle s_{\m_1},\dots,s_{\m_t}\rangle M_{k}\subseteq \sum\limits_{i=1}^t(s_{\m_i}M_{k_{\m_i}}) \subseteq M_n$ for all $n\geq k$.  Hence $M$ is Artinian.
\end{proof}

Let $R$ be a ring and $S$ a multiplicative subset of $R$. Recall from \cite{ONTK21}, an $R$-module $M$ is said to be $S$-cofinite (called finitely $S$-cogenerated in \cite[Definition 3]{ONTK21}) if  for each nonempty family of submodules $\{M_i\}_{i\in \Delta}$ of $M$, $\bigcap\limits_{i\in \Delta}M_i=0$ implies that $s(\bigcap\limits_{i\in \Delta'}M_i)=0$ for some $s\in S$ and a  finite subset $\Delta'\subseteq \Delta$.
If $S=\{1\}$, then $S$-cofinite modules are exactly the classical cofinite modules.

\begin{definition} Let $R$ be a ring and $S$ a multiplicative subset of $R$.  An $R$-module $M$ is called $u$-$S$-cofinite $($with respect to $s)$ if there is an $s\in S$ such that for each nonempty family of submodules $\{M_i\}_{i\in \Delta}$ of $M$, $\bigcap\limits_{i\in \Delta}M_i=0$ implies that $s(\bigcap\limits_{i\in \Delta'}M_i)=0$ for a  finite subset $\Delta'\subseteq \Delta$.
\end{definition}

Obviously, cofinite $R$-modules are $u$-$S$-cofinite, and $u$-$S$-cofinite $R$-modules is $S$-cofinite.

\begin{proposition} \label{s-loc-u-cofini}
	Let $R$ be a ring,  $M$ an $R$-module and $S$ a multiplicative subset of $R$. Then the following statements hold.
	\begin{enumerate}
		\item If $S_1\subseteq S_2$ are  multiplicative subsets of $R$ and $M$ is $u$-$S_1$-cofinite. Then  $M$ is $u$-$S_2$-cofinite.
		\item Suppose $S^{\ast}$ is the saturation of $S$. Then $M$ is $u$-$S$-cofinite if and only if $M$ is $u$-$S^{\ast}$-cofinite.
	\end{enumerate}
\end{proposition}
\begin{proof} (1) is trivial.
	
	(2) If $M$ is $u$-$S$-cofinite, then $M$ is also $u$-$S^{\ast}$-cofinite by (1). Suppose $M$ is
	$u$-$S^{\ast}$-cofinite. Then there is an $r\in S^{\ast}$ such that for each nonempty family of submodules $\{M_i\}_{i\in \Delta}$ of $M$, $\bigcap\limits_{i\in \Delta}M_i=0$ implies that $r(\bigcap\limits_{i\in \Delta'}M_i)=0$ for a  finite subset $\Delta'\subseteq \Delta$. Let $s:=rt\in S$ for some $t\in R$. Then $sr(\bigcap\limits_{i\in \Delta'}M_i)=s0=0$. Hence  $M$ is $u$-$S$-cofinite.
\end{proof}

Let $\p$ be a prime ideal of $R$. We say an $R$-module $M$ is \emph{$u$-$\p$-cofinite} provided that  $M$ is  $u$-$(R\setminus\p)$cofinite. The next result gives a local characterization of cofinite modules.
\begin{proposition}\label{s-noe-m-loc-char}
	Let $R$ be a ring and $M$ an $R$-module. Then the following statements are equivalent:
	\begin{enumerate}
		\item  $M$ is cofinite;
		\item   $M$ is $u$-$\p$-cofinite for any $\p\in \Spec(R)$;
		\item   $M$ is  $u$-$\m$-cofinite for any $\m\in \Max(R)$.
	\end{enumerate}
\end{proposition}

\begin{proof}  $(1)\Rightarrow (2)\Rightarrow (3)$: Trivial.
	
	$(3)\Rightarrow (1)$: Let $\{M_i\}_{i\in \Delta}$ be a family of submodules of $M$ such that  $\bigcap\limits_{i\in \Delta}M_i=0$.   Then , for each $\m\in\Max(R)$, there exist $s_\m\not\in \m$ and $k_{\m}\in \mathbb{Z}^{+}$ such that  $s_\m(\bigcap\limits_{i\in \Delta'_\m}M_i)=0$ for a  finite subset $\Delta'_\m\subseteq \Delta$. Since $R$ can be generated by a finite subset $\{s_{\m_1},\dots,s_{\m_t}\}$.
	Let $\Delta'=\bigcap\limits_{i=1}^t\Delta'_{\m_i}$. Then $\bigcap\limits_{i\in \Delta'}M_i=\langle s_{\m_1},\dots,s_{\m_t}\rangle(\bigcap\limits_{i\in \Delta'}M_i)\subseteq \sum\limits_{i=1}^t (s_{\m_i}\bigcap\limits_{i\in \Delta'_{\m_i}}M_i)=0$. Hence, $M$ is cofinite.
\end{proof}

\begin{definition} Let $\N$ be a nonempty family of submodules of $M$. Then $N\in\N$ is called an $S$-minimal element of $\N$ with respect to $s$  if  whenever $N'\subseteq N$ for some $N\in \N$ then $sN\subseteq N'$. We say $M$ satisfies $(S$-MIN$)$-condition with respect to $s$ if every nonempty family of submodules of $M$ has an $S$-minimal element  with respect to $s$.
\end{definition}


It follows from \cite[Proposition 10.10]{af12} that  an $R$-module $M$ is Artinian if and only if every factory $M/N$ is finitely  cogenerated, if and only if $M$ satisfies (MIN)-Condition. Recently,  \"{O}zen extended this reusut to $S$-Artinian rings in \cite[Theorem 3]{ONTK21}.   Now we give a uniform $S$-version of  \cite[Proposition 10.10]{af12}.

\begin{theorem} \label{s-loc-u-cofini}
	Let $R$ be a ring,   $S$ a multiplicative subset of $R$ and $M$ an $R$-module. Let $s\in S$. Then the following statements are equivalent:
	\begin{enumerate}
		\item $M$ is a $u$-$S$-Artinian module with respect to $s$;
		\item $M$ satisfies $(S$-MIN$)$-condition with respect to $s$;
		\item  For any nonempty family $\{N_i\}_{i\in \Gamma}$ of submodules of $M$, there is a finite subset $\Gamma_0\subseteq \Gamma$ such that $s\bigcap\limits_{i\in \Gamma_0}N_i\subseteq \bigcap\limits_{i\in \Gamma}N_i$;
		\item    Every factor module $M/N$  is $u$-$S$-cofinite with respect to $s$.
	\end{enumerate}
\end{theorem}

\begin{proof}$(1)\Rightarrow (2):$ Suppose that $M$ is a $u$-$S$-Artinian module with respect to $s$. Let $\N$ be a nonempty set  of  submodules of $M$. On contrary, suppose $\N$ has no $S$-minimal element of $\N$ with respect to $s$. Take $N_1\in \N$, and then there exists $N_2\in\N$ such that $N_1\supseteq N_2$ and $sN_1\nsubseteq N_2$. Iterating these steps, we can obtain a descending chain $N_1\supseteq N_2\supseteq \cdots\supseteq N_n\supseteq \cdots$ such that $sN_k\nsubseteq N_{k+1}$ for any $k$.  This  implies $M$ is not a
	$u$-$S$-Artinian module with respect to $s$, which is a contradiction.
	
	$(2)\Rightarrow (3):$ Let  $\{N_i\}_{i\in \Gamma}$ be a nonempty family of submodules of $M$. Set $N=\bigcap\limits_{i\in \Gamma}N_i$. Let $\A$ be the set of all intersections of finitely many $N_i$. Then each $N_i\in \A$, and so $\A$ is nonempty. So there is an $S$-minimal element, say $A=\bigcap\limits_{i\in \Gamma_0}N_i$, of $\A$. It is clear that $N\subseteq A$. For each $i\in \Gamma$, we have $sA\subseteq A\cap N_i\subseteq N_i$ by the  $S$-minimality of $A$ with respect to $s$. So $sA\subseteq \bigcap\limits_{i\in \Gamma}N_i=N$.

	$(3)\Rightarrow (1):$ Let $N_1\supseteq N_2\supseteq \cdots\supseteq N_n\supseteq \cdots$ be a descending chain of submodules of $M$. Then there is positive inter $k$ such that $sN_k=s\bigcap\limits_{i=1}^kN_i \subseteq \bigcap\limits_{i=1}^{\infty}N_i$.  So $sN_k\subseteq N_n$ for any $n\geq k$.
	
	$(3)\Rightarrow (4):$ Suppose $\bigcap\limits_{i\in \Gamma}N_i/N=0$ for some family of submodules $\{N_i/N\}_{i\in \Gamma}$ of $M/N$. Then  $\bigcap\limits_{i\in \Gamma}N_i=N$. Set $\A=\{\bigcap\limits_{i\in \Gamma'}N_i\mid \Gamma'\subseteq \Gamma $ is a finite subset $\}.$ By $(3)$, $\A$ has an $S$-minimal element with respect to $s$, say $M_0=\bigcap\limits_{i\in \Gamma_0}N_i$ for some finite subset $\Gamma_0\subseteq \Gamma$. So, for any  $k\in \Gamma-\Gamma_0$, we have $sM_0\subseteq M_0\bigcap N_k$. Thus $sM_0\subseteq  M_0\cap (\bigcap\limits_{k\in \Gamma-\Gamma_0}N_k)\subseteq \bigcap\limits_{k\in \Gamma}N_k=N$. Consequently, $M/N$ is $u$-$S$-cofinite with respect to $s$.

	$(4)\Rightarrow (1):$  Let $N_1\supseteq N_2\supseteq \cdots\supseteq N_n\supseteq \cdots$ be a descending chain of submodules of $M$. Set $N=\bigcap\limits_{i=1}^{\infty}N_i$. By assumption, $M/N$ is is $u$-$S$-cofinite with respect to $s$. Note that $\bigcap\limits_{i=1}^{\infty}N_i/N=0$ in $M/N$. Then there is a positive integer $k$ such that $s\bigcap\limits_{i=1}^{k}N_i/N=0$ by $(4)$. So $sN_k\subseteq N_n$ for all $n\geq k$. So $M$ is a $u$-$S$-Artinian module with respect to $s$.
\end{proof}

\section{basic properties of $u$-$S$-Artinian  rings}

Recall from \cite[Definition 2.1]{stk20} that a ring $R$ is called an $S$-Artinian ring if any descending chain of ideals $\{I_i\}_{i\in \mathbb{Z}^{+}}$ of $R$ is $S$-stationary with respect to  some $s\in S$. Note that the $s$ is determined by the descending chain of ideals in the definition of  $S$-Artinian rings. Now we introduce a ``uniform'' version of $S$-Artinian rings.

\begin{definition} Let $R$ be a ring and $S$ a multiplicative subset of $R$.  Then $R$ is called a
	$u$-$S$-Artinian $($abbreviates uniformly $S$-Artinian$)$ ring $($with respect to $s)$ provided that $R$ is a $u$-$S$-Artinian $R$-module $($with respect to $s)$, that is, there exists $s\in S$ such that each descending chain $\{I_i\}_{i\in \mathbb{Z}^{+}}$ of ideals  of $R$ is $S$-stationary with respect to $s$.
\end{definition}

Since  $u$-$S$-Artinian rings are $u$-$S$-Artinian modules over themselves, the results in Secton 1 also hold for $u$-$S$-Artinian rings. Specially,  $u$-$S$-Artinian rings are  $S$-Artinian. However, $S$-Artinian rings are not  $u$-$S$-Artinian in general. A counterexample was given in  Example \ref{exam-not-ut-1}. If $0\in S$, then every ring $R$ is $u$-$S$-Artinian. So Artinian rings are not $u$-$S$-Artinian in general.  A  multiplicative set $S$ is said to be  regular if every element in $S$ is a non-zero-divisor. The following Proposition shows that  $u$-$S$-Artinian rings are exactly Artinian for any regular multiplicative set $S$.

\begin{proposition}\label{s-artinian-regular} Let $R$ be a ring and  $S$ a regular multiplicative subset of $R$. If $R$ is a $u$-$S$-Artinian ring, then $R$ is an Artinian ring.
\end{proposition}
\begin{proof} Let $s$ be an element in $S$. Consider the descending chain $Rs\subseteq Rs^2\subseteq \cdots$ of ideals of $R$. Then there exists $k$ such that $sRs^k\subseteq Rs^n$ for any $n\geq k$. In particular, we have  $s^{k+1}=rs^{k+2}$ for some $r\in R$. Since $s$ is a non-zero-divisor, we have $1=rs$, and thus $s$ is a unit. So $R$ is an Artinian ring.
\end{proof}

In order to give a non-trivial $u$-$S$-Artinian ring which is not  Artinian, we consider the direct product of  $u$-$S$-Artinian ring.

\begin{proposition}\label{dird-s-semp} Let $R = R_1\times R_2$ be direct product of rings $R_1$ and $R_2$, and $S = S_1\times S_2$ a direct product of multiplicative subsets of $R_1$ and $R_2$. Then $R$ is a $u$-$S$-Artinian ring if and only if $R_i$ is a $u$-$S_i$-Artinian ring for each $i=1, 2$.
\end{proposition}
\begin{proof}  Suppose $R$ is a $u$-$S$-Artinian ring with respect to $s_1\times s_2$. Let $\{I^1_i\}_{i\in \mathbb{Z}^{+}}$ be a descending chain of ideals of $R_1$. Then $\{I^1_i\times 0\}_{i\in \mathbb{Z}^{+}}$ be a descending chain of ideals of $R$. Then there exists an integer $k$ such that  $(s_1\times s_2) (I^1_k\times 0)\subseteq I^1_n\times 0$ for all $n\geq k$. Hence $s_1 I^1_k\subseteq I^1_n$ for all $n\geq k$. So $R_1$ is a $u$-$S_1$-Artinian ring with respect to $s_1$. Similarly, $R_2$ is a $u$-$S_2$-Artinian ring with respect to $s_2$.  On the other hand, suppose  $R_i$ is a $u$-$S_i$-Artinian ring with respect to $s_i$ for each $i=1, 2$. Let $\{I_i=I^1_i\times I^2_i\}_{i\in \mathbb{Z}^{+}}$ be a descending chain of ideals of $R$. Then there exists an integer $k_i$ such that $s_i I^i_{k_i}\subseteq I^i_n$ for all $n\geq k_i$. Set $k=\max\{k_1,k_2\}$. Then $(s_1\times s_2) I_k=s_1I^1_k\times s_2I^2_k \subseteq I^1_n\times I^2_n= I_n$ for all $n\geq k$. So  $R$ is a $u$-$S$-Artinian ring.
\end{proof}

The promised non-Artinian $u$-$S$-Artinian rings are given as follows.
\begin{example}\label{exam-not-ut-2}
	Let $R=R_1\times R_2$ be a product of $R_1$ and $R_2$, where $R_1$ is an Artinian ring but  $R_2$ is not Artinian. Set $S=\{1\}\times\{1,0\}$. Then $R$ is a $u$-$S$-Artinian ring but not Artinian by Proposition \ref{dird-s-semp}.
\end{example}

Let $R$ be a ring and  $S$ a multiplicative subset of $R$. Recall from \cite{zwz21-p} that an $R$-module $M$ is called $u$-$S$-semisimple (with respect to $s$) provided that any $u$-$S$-short exact sequence $0\rightarrow A\xrightarrow{f} M\xrightarrow{g} C\rightarrow 0$ is $u$-$S$-split (with respect to $s$), i.e., there exists an $R$-homomorphism $h:B\rightarrow A$ such that $h\circ f=s\Id_A$ for some $s\in S$. A ring  $R$ is called a $u$-$S$-semisimple ring if every free $R$-module is $u$-$S$-semisimple. The rest of this section is devoted to show any $u$-$S$-semisimple ring is $u$-$S$-Artinian. First, we introduce the notion of $u$-$S$-simple modules.

\begin{definition}
	An $R$-module $M$ is said to be $u$-$S$-simple $($with respect to $s)$ provided that $M$ is not $u$-$S$-torsion with respect to some $s\in S$,  and any proper submodule of $M$ is $u$-$S$-torsion with respect to $s$.
\end{definition}

Since any proper submodule of a $u$-$S$-simple $R$-module is  $u$-$S$-torsion, we have any  $u$-$S$-simple $R$-module is $u$-$S$-semisimple by \cite[Lemma 2.1]{zwz21-p}. Moreover, we have the following result.

\begin{proposition}\label{ussi-usse}
	Suppose $M$ is a $u$-$S$-simple $R$-module. Then $M^{(\aleph)}$ is a $u$-$S$-semisimple $R$-module for any ordinal $\aleph$.
\end{proposition}
\begin{proof} Suppose $M$ is a $u$-$S$-simple $R$-module with respect to $s$ and $N$ is a submodule of $M^{(\aleph)}$ such that $M^{(\aleph)}/N$ is not $u$-$S$-torsion with respect to $s$. Set $\Gamma=\{\alpha\subseteq\aleph\mid s(N\cap M^{(\alpha)})=0\}$. For each  $i\in \aleph$, we set $ M^{i}$ to be the  $i$-th component  of $M^{(\aleph)}$. Then we claim there is $i\in \aleph$ such that $s(N\cap M^{i})=0$, and hence $\Gamma$ is not empty. Indeed, on contrary, suppose $N\cap M^{i}$ is  not $u$-$S$-torsion for any $i\in \aleph$. Since $M$ is $u$-$S$-simple with respect to $s$, then $N\cap  M^{i}= M^{i}\cong M$, and hence $N=M^{(\aleph)}$ which is a contradiction. Let $\Lambda$ be a chain in $\Gamma$. Then  $s(N\cap \bigcup\limits_{\alpha\in\Lambda}M^{(\alpha)})=\bigcup\limits_{\alpha\in\Lambda} s(N\cap M^{(\alpha)})=0$. So $\Lambda$ has a upper bound. By Zorn Lemma, there  is a maxiaml element, say $\beta$, in $\Gamma$. Set $L=M^{(\beta)}$.  We claim that $M^{(\aleph)}$ is $u$-$S$-isomorphic to $N+L$. Otherwise, since $M^{j}\cong M$ is $u$-$S$-simple, there is an $M^{j}\not\subseteq N+L$ for some $j\in \aleph$. Hence $N\cap M^{(\beta\cup \{j\})}$ is $u$-$S$-torsion with respect to $s$, which contradicts the maximality of $\beta$. Since $N\cap L$ is $u$-$S$-torsion, $N$ is $u$-$S$-isomorphic to    $N/(N\cap L)$ and $(N+L)/(N\cap L)$ is $u$-$S$-isomorphic to $(N+L)$ which is also  $u$-$S$-isomorphic to $M^{(\aleph)}$.  Considering the split monomorphism $g:N/(N\cap L)\rightarrow (N+L)/(N\cap L)$, we have the natural embedding map $N\rightarrow M$ is also a  $u$-$S$-split monomorphism.
\end{proof}


\begin{theorem}\label{ussim-usart}
	Let $R$ be a ring and  $S$ a multiplicative subset of $R$. Suppose $R$ is a $u$-$S$-semisimple ring. Then $R$ is a $u$-$S$-artinian ring.
\end{theorem}
\begin{proof} Suppose $R$ is a $u$-$S$-semisimple ring. Let $\aleph$ be a cardinal larger than $2^{\sharp(R)}\cdot\aleph_0$, where $\sharp(R)$ is the cardinal of $R$. Then the free $R$-module $R^{(\aleph)}$ is  $u$-$S$-semisimple with respect to some $s\in S$. And so every subquotient of $R^{(\aleph)}$ is  also $u$-$S$-semisimple with respect to some $s\in S$. Let $I_1\supseteq I_2\supseteq \cdots\supseteq I_n\supseteq \cdots$ be a descending chain of ideals of $R$. Note that there are at most $2^{\sharp(R)}\cdot\aleph_0$ such chains. Set $R=I_0$. Consider the exact sequence $\xi_i:0\rightarrow I_i\rightarrow I_{i-1}\rightarrow I_{i-1}/I_i\rightarrow 0$ for any positive integer $i$. Since $R$ is a $u$-$S$-semisimple ring, then each $I_{i-1}/I_i$ is $u$-$S$-projective  by  \cite[Theorem 3.5]{zwz21-p}. So,  by \cite[Corollary 2.10]{zwz21-p},  each $\xi_i$ is $u$-$S$-split with respect to $s$. Hence, by \cite[Lemma 2.4]{zwz21-p}, there is a $u$-$S$-isomorphism $f_i: I_{i-1}\rightarrow I_i\oplus I_{i-1}/I_{i}$ with respect to $s$ for each $i$. So there  are  $u$-$S$-isomorphisms $$R\xrightarrow{f_1} I_1\bigoplus R/I_1\xrightarrow{f_2\oplus \Id} I_2\bigoplus I_1/I_2\bigoplus R/I_1\rightarrow \cdots\xrightarrow{f_{k}\oplus \Id} I_k\bigoplus(\bigoplus\limits_{i=0}^k I_i/I_{i-1})\rightarrow \cdots  $$
	Assume $f(1)\in \bigoplus\limits_{i=1}^k I_{i-1}/I_i\subseteq \bigoplus\limits_{i=1}^\infty I_{i-1}/I_i$ where $f=\lim\limits_{\rightarrow_k}((f_{k}\oplus \Id)\circ\cdots\circ  f_1)$. Then $R$ is $u$-$S$-isomorphic to $\bigoplus\limits_{i=0}^k I_i/I_{i-1}$ with respect to $s$. And so $I_k$ is $u$-$S$-torsion with respect to $s$. Hence $sI_k/I_{n}=0$, i.e., $sI_k\subseteq I_n$ for all $n\geq k$. So $R$ is a $u$-$S$-artinian ring.
\end{proof}

\begin{corollary}\label{ussim-usart-fgm}
	Let $R$ be a ring and  $S$ a multiplicative subset of $R$. Suppose $R$ is a $u$-$S$-semisimple ring. Then any $S$-finite $R$-module is $u$-$S$-artinian.
\end{corollary}
\begin{proof} Since the class of $u$-$S$-artinian modules is closed under $u$-$S$-isomorphisms, we just need to show any finitely generated $R$-module is $u$-$S$-artinian, which can easily be deduced by Proposition \ref{s-u-noe-s-exact} and  Theorem \ref{ussim-usart}.
\end{proof}

\section{a characterization of $u$-$S$-Artinian  rings}

It is well-known that a ring $R$ is Artinian if an only if $R$ is a Noetherian ring with its Jacobson radical $\Jac(R)$ nilpotent and $R/\Jac(R)$ a semisimple ring (see \cite[Theorem 4.1.10]{fk16}). In this section, we will give a``uniform'' $S$-version of this result. We begin with the notion of $u$-$S$-maximal submodules and the $u$-$S$-Jacobson radical of a given $R$-module.
\begin{definition}  Let $R$ be a ring,  $S$ a multiplicative subset of $R$ and $s\in S$.
	Then a submodule $N$ is said to be $u$-$S$-maximal in  an $R$-module $M$ with respect to $s$ provided that
	\begin{enumerate}
		\item $M/N$ is not $u$-$S$-torsion with respect to $s$;
		\item if $N\subsetneqq H\subseteq M$, then  $M/H$ is $u$-$S$-torsion with respect to $s$.
	\end{enumerate}
\end{definition}
Note that an $R$-module  $N$ is a $u$-$S$-maximal submodule of $M$ (with respect to $s$) if and only if $M/N$ is $u$-$S$-simple (with respect to $s$). If $M$ does not have any $u$-$S$-maximal submodule with respect to $s$, then we denote by $\Jac_s(M)=M$. Otherwise, we denote by $\Jac_s(M)$ the intersection of all $u$-$S$-maximal submodules of $M$ with respect to $s$. The submodule $\Jac_s(M)$ of $M$ is called the $u$-$S$-Jacobson radical of $M$  with respect to $s$.
\begin{definition}\label{usjaco}  Let $R$ be a ring,  $S$ a multiplicative subset of $R$ and $M$ an $R$-module. Then the $u$-$S$-Jacobson radical of $M$ is defined to be  $\Jac_S(M)=\bigcap\limits_{s\in S}\Jac_s(M)$ under the above notions.
\end{definition}

First, we have the following obversion.
\begin{lemma}\label{quot-jac}
	Let $M$ be an $R$-module. Then
	\begin{center}
		$\Jac_s(M/\Jac_s(M))=0$, and $\Jac_S(M/\Jac_S(M))=0$.
	\end{center}
\end{lemma}
\begin{proof} We just note that an $R$-module $N/\Jac_s(M)\subseteq M/\Jac_s(M)$ is $u$-$S$-maximal with respect to $s$ if and only if $N+\Jac_s(M)\subseteq M$ is $u$-$S$-maximal with respect to $s$.
\end{proof}

\begin{proposition}\label{jac0}
	Suppose $M$ is a $u$-$S$-Artinian $R$-module with $\Jac_S(M)$ $u$-$S$-torsion.
	Then there exists $T\subseteq M$ such that $sT=0$ and $M/T\subseteq \bigoplus\limits_{i=1}^tS_i$ where each $S_i$ is $u$-$S$-simple with respect to $s$ for some $s\in S$. Consequently, $M$ is a $u$-$S$-Noetherian $R$-module.
\end{proposition}
\begin{proof} If $M$  has no  $u$-$S$-maximal submodule, then we may assume $T=\Jac_S(M)=M$ is $u$-$S$-torsion. So the assertion  trivially holds. Now suppose $M$  has a  $u$-$S$-maximal submodule. Then the intersection of all $u$-$S$-maximal submodules of $M$ is  $u$-$S$-torsion. Since $M$ is $u$-$S$-Artinian, then there exists a finite family of $u$-$S$-maximal submodules, say $\{M_1,\dots,M_n\}$, of $M$ such that $T:=\bigcap\limits_{i=1}^n M_i$ is $u$-$S$-torsion with respect to some $t\in S$ by Theorem \ref{s-loc-u-cofini}. Note that $M/T$ is a submodule of $\bigoplus\limits_{i=1}^n M/M_i$, where $M/M_i$ is $u$-$S$-simple with respect to some $s_i\in S$. Set $s=ts_1\cdots s_n$. Then each $S_i:=M/M_i$ is $u$-$S$-simple with respect to $s$.  One can easily check  that  $M/T$, as a submodule of $\bigoplus\limits_{i=1}^n M/M_i$, is $u$-$S$-Noetherian with respect to $s$. Thus $M$ is $u$-$S$-Noetherian with respect to $s$.
\end{proof}

\begin{proposition}\label{jac1}
	Suppose $R$ is a $u$-$S$-Artinian ring. Then $R/\Jac_S(R)$ is a $u$-$S/\Jac_S(R)$-semisimple ring.
\end{proposition}
\begin{proof} Write $J=\Jac_S(R)$, $\overline{R}=R/J$ and $\overline{S}=S/\Jac_S(R)$. Since $R$ is a $u$-$S$-Artinian ring, $\overline{R}$ is a  $u$-$S$-Artinian $R$-module by Proposition \ref{s-u-noe-s-exact}. By Lemma \ref{quot-jac}, $\Jac_S(\overline{R})=0$. It follows from by Proposition \ref{jac0} that there exists an element $s\in S$ and an submodule  $T$ of $\overline{R}$ such that $sT=0$ and $\overline{R}/T\subseteq \bigoplus\limits_{i=1}^tS_i$ with each $S_i$  $u$-$S$-simple with respect to $s$. Now let $F=\overline{R}^{(\aleph)}$ be a free $\overline{R}$-module with $\aleph$ an arbitrary cardinal. Then there is a short exact sequence $0\rightarrow T^{(\aleph)}\rightarrow \overline{R}^{(\aleph)}\rightarrow (\overline{R}/T)^{(\aleph)}\rightarrow 0$. By Proposition \ref{jac0}, $ (\overline{R}/T)^{(\aleph)}$ is a submodule of $(\bigoplus\limits_{i=1}^tS_i)^{(\aleph)}\cong \bigoplus\limits_{i=1}^t(S_i)^{(\aleph)}$ which is $u$-$S$-semisimple by Proposition \ref{ussi-usse}. Hence $ (\overline{R}/T)^{(\aleph)}$ is also $u$-$S$-semisimple by \cite[Proposition 3.3]{zwz21-p}. Since $sT^{(\aleph)}=0$,  $ (\overline{R})^{(\aleph)}$ is a $u$-$S$-semisimple $R$-module. So $ (\overline{R})^{(\aleph)}$ is actually a $u$-$\overline{S}$-semisimple $\overline{R}$-module, that is, $\overline{R}$ is a $u$-$\overline{S}$-semisimple ring.
\end{proof}

\begin{lemma}\label{nakya-J}
	Let $R$ be a ring,  $S$ a multiplicative subset of $R$ and $S^{\ast}$ the saturation of $S$. Suppose $r\in R$ and $s\in S$. If $r-s\in \Jac_S(R)$, then $r\in S^{\ast}$.
\end{lemma}
\begin{proof} Assume on contrary  $r\not\in S^{\ast}$. Then we claim  $R/Rr$ is not $u$-$S$-torsion. Indeed, if $t(R/Rr)=0$ for some $t\in S$. Then $t=rr'$ for some $r'\in R$. So $r\in S^{\ast}$, which is a contradiction. We also claim that there exists a  $u$-$S$-maximal ideal $I$ of $R$ such that $r\in I$. Indeed, let $\Lambda$ be the set of ideals $J$ of $R$ that contains $r$ satisfying $R/J$ is not $u$-$S$-torsion. One can check that the union of any  ascending chain in $\Lambda$ is also in $\Lambda$. So there is a maximal element $I$ in $\Lambda$ by Zorn Lemma. Hence $I$ is a  $u$-$S$-maximal ideal of $R$.  Since
	$\Jac_S(R)\subseteq I$, we have $s\in I$. Then $s(R/I)=0$, which is a contradiction. So $r\in S^{\ast}$.
\end{proof}

\begin{proposition}\label{nakya-u} {\bf (Nakayama Lemma for $S$-finite modules)}
	Let $R$ be a ring,  $I\subseteq \Jac_S(R)$, $S$ a multiplicative subset of $R$  and  $M$ an $S$-finite  $R$-module. If $sM\subseteq IM$ for some $s\in S$, then $M$ is $u$-$S$-torsion.
\end{proposition}
\begin{proof}   Let $F$ be a finitely generated submodule, say generated by $\{m_1,\dots,m_n\}$, of $M$  satisfying $s'M\subseteq F$ and $sM\subseteq IM$ for some $s', s\in S$. Then $ss'F\subseteq IF$. So we can assume $M$ itself is  generated by $\{m_1,\dots,m_n\}$.
	Then we have $a:=s^n+a_1s^{n-1}+\cdots+a_{n-1}s+a_n=0$ where $a_i\in I^i$ by \cite[Theorem 2.1]{M89}. Note that $aM=0$ and $a-s^n\in I$. By Lemma \ref{nakya-J}, $a\in S^{\ast}$, where $S^{\ast}$ is the saturation of $S$. It follows that  there is $r\in R$ such that $ar\in S$. Hence $arM=0$, and so $M$ is $u$-$S$-torsion.
\end{proof}

Let $I$ be an ideal of $R$. we say $I$ is $S$-nilpotent if $sI^k=0$ for some integer $k$ and $s\in S$.

\begin{proposition}\label{jac-nil}
	Let $R$ be a ring and  $S$ a multiplicative subset of $R$. Suppose $R$ is a $u$-$S$-Artinian ring. Then $\Jac_S(R)$ is $S$-nilpotent.
\end{proposition}
\begin{proof}  Suppose $R$ is a $u$-$S$-Artinian ring with respect to some $t\in S$. Write $J=\Jac_S(R)$. Consider the descending chain $$J\supseteq J^2\supseteq \cdots \supseteq J^{n}\supseteq J^{n+1}\supseteq \cdots$$ Then there exists an integer $k$ such that $tJ^k\subseteq J^n$ for some $n\geq k$. We claim that $sJ^k=0$ for some $s\in S$. On contrary, set $\Gamma=\{I\subseteq R\mid sJ^kI\not=0 $ for all $s\in S\}$. Since $R\in \Gamma$, $\Gamma$ is non-empty. So there is an $S$-minimal element $I$ in $\Gamma$ by Theorem \ref{s-loc-u-cofini}.  Let $x\in I$ such that $sJ^kx\not=0$ for all $s\in S$. Then $0\not=stJ^kx\subseteq sJ^{k+1}x$. So $sJ^{k+1}x\not=0$ for any $s\in S$. Hence $Jx\in \Gamma$. Since $Jx\subseteq Rx\subseteq I$, there exists $s_1\in S$ such that  $s_1Rx\subseteq s_1I\subseteq Jx\subseteq Rx$ by the $S$-minimality of $I$ . So there exists $s_2\in S$ such that $s_2Rx=0$ by Proposition \ref{nakya-u}, which contradicts $sJ^kx\not=0$ for all $s\in S$.
\end{proof}

Recently, Qi et al. \cite[Definition 2.1]{QKWCZ21} introduced the notion of $u$-$S$-Noetherian rings. A ring $R$ is called a $u$-$S$-Noetherian (abbreviates uniformly $S$-Noetherian) ring  provided there exists an element $s\in S$ such that for any ideal $I$ of $R$, $sI \subseteq K$ for some finitely generated sub-ideal $K$ of $I$.
Finally, we will show the promised result.

\begin{theorem}\label{s-artinian-s-Noetherian} Let $R$ be a ring and  $S$ a multiplicative subset of $R$. Then the following statements are equivalent:
	\begin{enumerate}
		\item  $R$ is a $u$-$S$-Artinian ring;
		\item $R$ is a $u$-$S$-Noetherian ring, $\Jac_S(R)$ is  $S$-nilpotent and $R/\Jac_S(R)$ is a $u$- $S/\Jac_S(R)$-semisimple ring.
	\end{enumerate}
	
\end{theorem}
\begin{proof}
	$(1)\Rightarrow (2)$ Suppose  $R$ is a $u$-$S$-Artinian ring with respect to some $s\in S$. We just need to prove $R$ is $u$-$S$-Noetherian because the other two statements are showed in Proposition \ref{jac-nil} and Proposition \ref{jac1} respectively. Write $J=\Jac_S(R)$. By Proposition \ref{jac-nil}, there exits an integer $m$ such that $tJ^m=0$ for some $t\in S$. We will show $R$ is $u$-$S$-Noetherian by induction on $m$. Let $m=1$. It follows by Proposition \ref{jac0} that $R$ is $u$-$S$-Noetherian.  Now, let $m>1$. Set $\overline{R}=R/J^{m-1}$. Then $\overline{R}$ is also $u$-$S$-Artinian by Proposition \ref{s-u-noe-s-exact}. Note that  $\Jac_S(\overline{R})=J/J^{m-1}$. So $\Jac_S(\overline{R})^{m-1}$ is also $u$-$S$-torsion. Hence $\overline{R}$ is $u$-$S$-Noetherian by induction. Since $tJ^m=0$, $tJ^{m-1}$  can also be seen as an ideal of  $R/J$. Since $R/J$ is a $u$-$S/J$-semisimple ring,  $R/J$ is also $u$-$S/J$-Noetherian by \cite[Corollary 3.6]{zwz21-p}. So $tJ^{m-1}$, and hence $J^{m-1}$, are both $u$-$S$-Noetherian $R$-modules since $J^{m-1}$ is $u$-$S$-isomorphic to $tJ^{m-1}$. Considering the exact sequence $0\rightarrow J^{m-1}\rightarrow R\rightarrow R/J^{m-1}\rightarrow 0$, we have  $R$ is also $u$-$S$-Noetherian by \cite[Lemma 2.12]{QKWCZ21}.

	$(2)\Rightarrow (1)$ Write $J=\Jac_S(R)$. We may assume $R$ is $u$-$S$-Noetherian with respect to $s$ such that  $sJ^m=0$  and $R/J$ is a $u$-$S$-semisimple $R$-module with respect to $s$.  We claim that $J$ is a $u$-$S$-Artinian $R$-module. Since $sJ^{m-1}$ is an $S$-finite $R/J$-module and $R/J$ is a $u$-$S$-semisimple ring, we have $sJ^{m-1}$ is $u$-$S$-Artinian by Corollary \ref{ussim-usart-fgm}.  Consider the sequence $0\rightarrow sJ^{m-1}\rightarrow sJ^{m-2}\rightarrow sJ^{m-2}/sJ^{m-1} \rightarrow 0$. Since $sJ^{m-2}/sJ^{m-1}$ is an $S$-finite $R/J$-module, we have $sJ^{m-2}/sJ^{m-1}$ is $u$-$S$-Artinian by Corollary \ref{ussim-usart-fgm}. Since $sJ^{m-1}$ is  $u$-$S$-Artinian, $sJ^{m-2}$ is also $u$-$S$-Artinian by Proposition \ref{s-u-noe-s-exact}.  Iterating these steps, we have $sJ$ is also $u$-$S$-Artinian. So  $J$ is also $u$-$S$-Artinian since $J$ is $u$-$S$-isomorphic to $sJ$.

	Let $I_1\supseteq I_2\supseteq \cdots\supseteq I_n\supseteq \cdots$ be a descending chain of ideals of $R$.   Consider the following natural commutative diagram with exact rows: $$\xymatrix@R=20pt@C=25pt{
		0 \ar[r]^{}&I_i\cap J\ar@{^{(}->}[d]\ar[r]&I_i \ar[r]\ar@{^{(}->}[d]&(I_i+J)/J\ar[r] \ar@{^{(}->}[d] &0\\
		0 \ar[r]^{}&I_{i+1}\cap J\ar[r]&I_{i+1}  \ar[r]&(I_{i+1}+J)/J\ar[r] &0.\\}$$
	Since $J$ is $u$-$S$-Artinian, there is an integer $k_1$ such that $s(I_{k_1}\cap J)\subseteq I_n\cap J$ for $n\geq k_1$. Since $R/J$ is a $u$-$S/J$-semisimple ring,  $R/J$ is also a $u$-$S/J$-artinian ring by Theorem \ref{ussim-usart}.  Hence there is an integer $k_2$ such that $s((I_{k_2}+J)/J)\subseteq (I_n+J)/J$ for $n\geq k_2$. Taking $k=\max\{k_1,k_2\}$, we can easily deduce $sI_{k}\subseteq I_n$ for $n\geq k$. Hence $R$ is a $u$-$S$-Artinian ring.
\end{proof}

Let $R$ be a commutative ring and $M$ an $R$-module. Then the idealization of $R$ by $M$, denoted by $R(+)M$, is equal to $R\bigoplus M$ as $R$-modules with coordinate-wise addition and multiplication $(r_1,m_1)(r_2,m_2)=(r_1r_2,r_1m_2+r_2m_1)$. It is easy to verify that $R(+)M$ is a commutative ring with identity $(1,0)$(see \cite{DW09} for more details). Note that there is a natural exact sequence of $R(+)M$-modules:
$$0\rightarrow 0(+)M\rightarrow R(+)M\xrightarrow{\pi} R\rightarrow 0.$$ Let $S$ be a multiplicative subset of $R$. Then it is easy to verify that $S(+)M=\{(s,m)|s\in S, m\in M\}$ is a multiplicative subset of $R(+)M$ . Now, we give a $u$-$S$-Artinian property on  idealizations.

\begin{proposition}\label{trivial extension-usn} Let $R$ be a commutative ring, $S$ a multiplicative subset of $R$ and $M$  an $R$-module. Then $R(+)M$ is a $u$-$S(+)M$-Artinian ring if and only if $R$ is a
	$u$-$S$-Artinian ring and $M$ is an $S$-finite $R$-module.
\end{proposition}
\begin{proof} Suppose $R(+)M$ is a $u$-$S(+)M$-Artinian ring. Then $R\cong R(+)M/0(+)M$ is a $u$-$S$-Artinian ring essentially by Proposition \ref{s-u-noe-s-exact}. By Theorem \ref{s-artinian-s-Noetherian}, $R(+)M$ is a $u$-$S(+)M$-Noetherian ring. So $0(+)M$ is an  $S(+)M$-finite ideal of $R(+)M$, which implies that $M$ is an $S$-finite $R$-module.

	Suppose $R$ is a $u$-$S$-Artinian ring and $M$ is  an $S$-finite $R$-module. Then $M$ is $u$-$S$-Artinian $R$-module  by Proposition \ref{s-u-noe-s-exact}. Let $I^{\bullet}: I_1\supseteq I_2\supseteq \cdots $ be  an descending chain of ideals of $R(+)M$. Then there is  an descending chain of ideals of $R$: $\pi(I^{\bullet}): \pi(I_1)\supseteq \pi(I_2)\supseteq \cdots $, where $\pi:R(+)M\twoheadrightarrow R$ is the natural epimorphism. Thus there is an element $s'\in S$  which is independent of $I^{\bullet}$ satisfying that there exists $k'\in \mathbb{Z}^{+}$ such that  $s'\pi(I_{k'})\subseteq \pi(I_n)$ for any $n\geq k'$. Similarly, $I^{\bullet}\cap 0(+)M:   I_1\cap 0(+)M\supseteq I_2\cap 0(+)M\supseteq \cdots $ is  an descending chain of sub-ideals of $0(+)M$ which is equivalent to a descending chain of submodules of $M$. So there is an element $s''\in S$  satisfying that there exists $k''\in \mathbb{Z}^{+}$ such that $s''(I_{k''}\cap 0(+)M)\subseteq I_n\cap 0(+)M$ for any $n\geq k''$. Let $k=\max(k',k'')$ and $n\geq k$. Consider the following natural commutative diagram with exact rows:
	$$\xymatrix@R=20pt@C=25pt{
		0 \ar[r]^{}&I_n\cap 0(+)M \ar@{^{(}->}[d]\ar[r]&I_n \ar[r]\ar@{^{(}->}[d]&\pi(I_n)\ar[r] \ar@{^{(}->}[d] &0\\
		0 \ar[r]^{}&I_k\cap 0(+)M \ar[r]&I_k \ar[r]&\pi(I_k) \ar[r] &0.\\}$$
	Set $s=s's''$. Then we have $sI_k\subseteq I_n$ for any $n\geq k$. So $R(+)M$ is a $u$-$S(+)M$-Artinian ring.
\end{proof}

Taking $S=\{1\}$ in Proposition \ref{trivial extension-usn}, we can easily deduce the following classical result.
\begin{corollary} Let $R$ be a commutative ring and $M$  an $R$-module. Then $R(+)M$ is an Artinian ring if and only if $R$ is an Artinian ring and $M$ is a finitely generated $R$-module.
\end{corollary}

\begin{acknowledgement}\quad\\
	The second author was supported by National Natural Science Foundation of China (No. 12201361).
\end{acknowledgement}


\begin{thebibliography}{99}
	
	
	\bibitem{Ah18} H. Ahmed, On $S$-Mori domains, {\it  J. Algebra Appl.} {\bf17}(9) (2018) 1850171, 11 pp.
	
	\bibitem{ad02}  D. D.  Anderson and T.  Dumitrescu,  $S$-Noetherian rings, {\it Comm. Algebra} {\bf 30}(9) (2002) 4407-4416.
	
	\bibitem{ahz19} D. D. Anderson, A. Hamed and M. Zafrullah, On $S$-GCD domains, {\it J. Algebra Appl.} {\bf 18}(4)(2019) 1950067, 14 pp.
	
	\bibitem{af12}  D. W.  Anderson and K. R. Fuller,  Rings and Categories of Modules, 13, Springer Science \& Business Media, 2012.
	
	\bibitem{DW09}  D. D. Anderson and M.  Winders,  Idealization of a module, {\it J. Commut. Algebra} \textbf{1}(1) (2009) 3-56.
	
	
	\bibitem{B19}  S. Bazzoni and L. Positselski, $S$-almost perfect commutative rings, {\it J. Algebra} \textbf{532} (2019) 323-356.
	
	\bibitem{bh18} D. Bennis and  M. El Hajoui,   On $S$-coherence, {\it J. Korean Math. Soc.} {\bf55}(6) (2018)  1499-1512.
	
	
	
	
	
	
	\bibitem{L01} L. Fuchsand L. Salce, Modules over non-Noetherian domains, volume 84 of Mathematical Surveys and Monographs. American Mathematical Society, Providence, RI, 2001.
	
	\bibitem{kkl14} H. Kim,  M. O. Kim and J. W.  Lim,     On $S$-strong Mori domains, {\it J. Algebra}   {\bf416}  (2014)  314-332.
	\bibitem{M89}  H. Matsumura,   Commutative Ring Theory, volume 8 of Cambridge Studies in Advanced Mathematics.
	Cambridge University Press, Cambridge, 2nd edn (1989). Translated from the Japanese by M. Reid
	
	\bibitem{OA21}   K. N. Omid and H. Ahmed,  Weakly $S$-artinian modules, {\it Filomat}, {\bf35}(15) (2021)  5215-5226.
	
	\bibitem{ONTK21} M. \"{O}zen,  O. Naji, U. Tekir and K. Shum, Characterization Theorems of $S$-Artinian Modules, {\it  Comptes rendus de l'Acad\'{e}mie bulgare des sciences: sciences math\'{e}matiques et naturelles} {\bf74} (2021) 496-505.
	
	
	
	
	
	
	
	
	
	
	\bibitem{QKWCZ21} W. Qi,\  H. Kim,\ F. G. Wang,\ M. Z. Chen and\ W. Zhao, Uniformly $S$-Noetherian rings, https://arxiv.org/abs/2201.07913.
	
	\bibitem{stk20} E. S. Sengelen, U. Tekir and S. Koc,  On $S$-Artinian rings and finitely $S$-cogenerated rings,{\it J. Algebra Appl.} {\bf19}(3) (2020) 2050051 (16 pages).
	
	\bibitem{fk16} F. G.  Wang and  H. Kim, Foundations of Commutative Rings and Their Modules, Singapore, Springer, 2016.
	
	\bibitem{zwz21}  X. L. Zhang,  Characterizing $S$-flat modules and $S$-von Neumann regular rings by uniformity,   {\it Bull. Korean Math. Soc.}, {\bf 59}(3) (2022) 643-657.
	
	
	
	\bibitem{zwz21-p}  X. L. Zhang and Wei Qi, Characterizing $S$-projective modules and $S$-semisimple rings by uniformity, {\it J. Commut. Algebra}, to appear.\\ https://projecteuclid.org/journals/jca/journal-of-commutative-algebra/acceptedpapers
\end{thebibliography}
\end{document}